\documentclass[a4paper]{amsart}

\usepackage{amsfonts}
\usepackage{amsmath}
\usepackage{amssymb}
\usepackage{amsthm}
\newtheorem{theorem}{Theorem}
\newtheorem{prop}{Proposition}
\newtheorem{lemma}{Lemma}
\newtheorem{cor}{Corollary}

\theoremstyle{definition}{
\newtheorem{rem}{Remark}
}

\newcommand{\cG}{{\mathcal G}}

\DeclareMathOperator{\diag}{diag} 

\DeclareMathOperator{\GL}{GL} 

\DeclareMathOperator{\In}{I}

\begin{document}
\author{Hans Havlicek \and Andrzej Matra\'s \and Mark Pankov}
\title{Geometry of free cyclic submodules over ternions}

\maketitle


\begin{abstract}
Given the algebra $T$ of ternions (upper triangular $2\times 2$ matrices) over
a commutative field $F$ we consider as set of points of a projective line over
$T$ the set of all free cyclic submodules of $T^2$. This set of points can be
represented as a set of planes in the projective space over $F^6$. We exhibit
this model, its adjacency relation, and its automorphic collineations. Despite
the fact that $T$ admits an $F$-linear antiautomorphism, the plane model of our
projective line does not admit any duality.
\end{abstract}

\section{Introduction}

Note that all our rings are associative, with a unit element $1\neq 0$, which
acts unitally on modules, and is inherited by subrings.

One of the crucial tasks in ring geometry is to find a ``good'' definition of
the \emph{projective line} over a ring $R$. In terms of left homogeneous
coordinates a \emph{point} of such a line should be a \emph{cyclic submodule}
$R(a,b)$ of the free $R$-left module $R^2$. But which pairs $(a,b)$ should be
allowed to be generators of points? Indeed, there are different definitions and
we refer to \cite{lash-97} and \cite{veld-85} for detailed discussion which
includes also ``higher-dimensional'' spaces.

The aim of the present paper is to exhibit the interplay between two notions:
On the one hand we consider as \emph{points\/} the \emph{free cyclic
submodules\/} $R(a,b)$ of $R^2$. On the other hand we have the distinguished
subset of \emph{unimodular points}, \emph{i.~e.}, points of the form $R(a,b)$
such that there are $x,y\in R$ with $ax+by=1$. Several authors consider as
points only our unimodular points. Often also extra conditions on the
coordinate ring (like $R$ being of stable rank $2$) can be found; see
\cite{blunck+he-05}, \cite{herz-95}, or \cite{veld-85}. At the other extreme,
in \emph{projective lattice geometry\/} any cyclic (or: $1$-generated)
submodule $R(a,b)\subset R^2$ is called a point, whereas our points are called
\emph{free points} in this context \cite[p.~1129]{bre+g+s-95}; see also
\cite{faure-04a}.

For any ring $R$, the submodule $R(1,0)$ is a unimodular point, but the
existence of non-unimodular points cannot be guaranteed. An easy example is the
projective line over any field. All its points are unimodular. So it seems
necessary to restrict ourselves to a class of rings for which non-unimodular
points do exist. One such class is formed by the algebras of ternions over
commutative fields. There are several articles which describe the geometry over
ternions based on unimodular points \cite{beck-36b}, \cite{beck-36a},
\cite{benz-77}, \cite{depu-59}, \cite{depu-60} in great detail, but little
seems to be known about the properties of the remaining (non-unimodular) points
\cite{havlicek+saniga-09a}, \cite{saniga+h+p+p-08a}, \cite{saniga+pracna-08a}.
\par
Results and notions which are used without further reference can be found, for
example, in \cite{pankov-10a}.

\section{Cyclic submodules}

Let $F$ be a (commutative) field and $T$ be the ring of \emph{ternions},
\emph{i.~e.}, upper triangular $2\times 2$ matrices
\begin{equation*}
    \begin{pmatrix} x & y\\0 & z\end{pmatrix},
\end{equation*}
over $F$. We shall often identify $x\in F$ with the ternion $xI$, where $I\in
T$ denotes the $2\times 2$ identity matrix. In terms of this identification $F$
equals the centre of $T$. The ring $T$ is a non-commutative
\emph{three-dimensional\/} algebra over $F$, a fact which is reflected in its
name\footnote{According to \cite{benz-77} the term ``ternions'' is due to
E.~Study (1889). It is worth noting that J.~Petersen used the same phrase in a
different meaning already in 1885, namely as a name for a three-dimensional
commutative algebra over the real numbers \cite[3.2]{luetzen-01a}.}. A nice
algebraic characterisation of the algebra of ternions over $F$ can be found in
\cite{lex+p+w-80a}.
\par
By \cite[Lemma~2]{havlicek+saniga-09a}, the non-zero cyclic submodules of the
free $T$-left module $T^{2}$ fall into five orbits under the action of the
group $\GL_2(T)$. Below we give one representative for each orbit:
\begin{equation}\label{eq:X0}
    X_{0}:= T \left[
        \begin{pmatrix} 1 & 0\\   0 & 1\end{pmatrix},
        \begin{pmatrix} 0 & 0\\   0 & 0\end{pmatrix}
        \right]=
        \left\{\left.\left[
        \begin{pmatrix} x & y\\0 & z\end{pmatrix},
        \begin{pmatrix} 0 & 0\\0 & 0\end{pmatrix}
        \right]\right|x,y,z\in F\right\},
\end{equation}
\begin{equation}\label{eq:Y0}
    Y_{0}:=
        T \left[
        \begin{pmatrix} 0 & 0\\0 & 1\end{pmatrix},
        \begin{pmatrix} 0 & 1\\0 & 0\end{pmatrix}
        \right] =
        \left\{\left.\left[
        \begin{pmatrix} 0 & y\\0 & z\end{pmatrix},
        \begin{pmatrix} 0 & x\\0 & 0\end{pmatrix}
        \right]\right| x,y,z\in F \right\},
\end{equation}
\begin{equation}\label{eq:alpha0}
    \alpha_{0}:=
    T \left[
    \begin{pmatrix} 0 & 0\\0 & 1\end{pmatrix},
    \begin{pmatrix} 0 & 0\\0 & 0\end{pmatrix}
    \right] =
    \left\{\left.\left[
    \begin{pmatrix} 0 & y\\0 & z\end{pmatrix},
    \begin{pmatrix} 0 & 0\\0 & 0\end{pmatrix}
    \right]\right| y,z\in F \right\},
\end{equation}
\begin{equation}\label{eq:beta0}
    \beta_{0}:=
        T \left[
        \begin{pmatrix} 1 & 0\\   0 & 0    \end{pmatrix},
        \begin{pmatrix} 0 & 0\\   0 & 0    \end{pmatrix}
        \right] =
        \left\{\left.\left[
        \begin{pmatrix} x & 0\\   0 & 0    \end{pmatrix},
        \begin{pmatrix} 0 & 0\\   0 & 0    \end{pmatrix}
        \right]\right| y,z\in F \right\},
\end{equation}
\begin{equation}\label{eq:gamma0}
    \gamma_{0}:=
    T \left[
    \begin{pmatrix} 0 & 1\\0 & 0\end{pmatrix},
    \begin{pmatrix} 0 & 0\\0 & 0\end{pmatrix}
    \right] =
    \left\{\left.\left[
    \begin{pmatrix} 0 & x\\0 & 0\end{pmatrix},
    \begin{pmatrix} 0 & 0\\0 & 0\end{pmatrix}
    \right]\right| x\in F \right\}.
\end{equation}
An arbitrary submodule from the orbit of $X_0$ has the form $X=X_{0}S$ with
$S\in \GL_{2}(T)$, and will be called an \emph{$X$-submodule} for short.
Similarly the other types of cyclic submodules are called
\emph{$Y$-submodules}, \emph{$\alpha$-submodules}, and so on. Submodules of the
first two types are free, whence they are elements of the set of points. The
unimodular points are the $X$-submodules, the non-unimodular points are the
$Y$-submodules of $T^2$. The remaining three types are torsion.
\par
Any $2\times 2$ matrix $S$ over $T$ can be considered as a $4\times 4$ matrix
over $F$ of the block form
\begin{equation}\label{eq:block}
    \left(\!\!
    \begin{array}{cc|cc}
    a_{11} & a_{12} & b_{11} & b_{12}\\
    0      & a_{22} & 0      & b_{22}\\
    \hline
    c_{11} & c_{12} & d_{11} & d_{12}\\
    0      & c_{22} & 0      & d_{22}
    \end{array}\!\!\right).
\end{equation}
We have $S\in \GL_{2}(T)$ if, and only if, its determinant (as matrix over $F$)
satisfies
\begin{equation}\label{eq:det4}
    \det S =
    ( a_{22} d_{22}-b_{22} c_{22})
    ( a_{11} d_{11}-b_{11} c_{11})\neq 0,
\end{equation}
because invertibility of $S$ over $F$ implies that $S^{-1}$ can be partitioned
into four ternions as in (\ref{eq:block}). Hence it is easy to determine
whether or not a $2\times 2$ matrix over $T$ is invertible or not. See
\cite{benz-79} and \cite[pp.~9--10]{depu-60} for more results on invertibility
and the actual inversion of matrices over ternions.
\par
By assuming $S$ to be invertible, we obtain from
\eqref{eq:X0}--\eqref{eq:gamma0} and \eqref{eq:det4} the following general form
for $X$-submodules, $Y$-submodules, $\alpha$-submodules, and so on:
\begin{equation}\label{eq:X}
        \begin{array}{l}
        X =   \left\{\left.\left[
        \begin{pmatrix}
        a_{11}x & a_{12}x+ a_{22}y\\
        0       & a_{22}z
        \end{pmatrix},
        \begin{pmatrix}
        b_{11}x & b_{12}x+ b_{22}y\\
        0       & b_{22}z
        \end{pmatrix}
        \right]\right|x,y,z\in F \right\}\\
        \quad\quad\quad\quad\mbox{with}\;\;(a_{11},b_{11}),(a_{22},b_{22})\neq (0,0),
        \end{array}
\end{equation}
\begin{equation}\label{eq:Y}
        \begin{array}{l}
        Y = \left\{\left.\left[
        \begin{pmatrix}
        0 & a_{22}y+ c_{22}x\\
        0 & a_{22}z
        \end{pmatrix},
        \begin{pmatrix}
        0 & b_{22}y+ d_{22}x\\
        0 & b_{22}z
        \end{pmatrix}
        \right]\right|x,y,z\in F \right\}\\
        \quad\quad\quad\quad\mbox{with}\;\;( a_{22} d_{22}-b_{22} c_{22})\neq 0,
        \end{array}
\end{equation}
\begin{equation}
        \label{eq:alpha} \alpha =   \left\{\left.\left[
        \begin{pmatrix}
        0 & a_{22}y\\
        0 & a_{22}z
        \end{pmatrix},
        \begin{pmatrix}
        0 & b_{22}y\\
        0 & b_{22}z
        \end{pmatrix}
        \right]\right|y,z\in F \right\}
        \;\;\mbox{with}\;\;(a_{22},b_{22})\neq (0,0),
\end{equation}
\begin{equation}
        \label{eq:beta} \beta =   \left\{\left.\left[
        \begin{pmatrix}
        a_{11}x & a_{12}x\\
        0 & 0
        \end{pmatrix},
        \begin{pmatrix}
        b_{11}x & b_{12}x\\
        0 & 0
        \end{pmatrix}
        \right]\right|x\in F \right\}
        \;\;\mbox{with}\;\;(a_{11},b_{11})\neq (0,0),
\end{equation}
\begin{equation}
        \label{eq:gamma} \gamma =   \left\{\left.\left[
        \begin{pmatrix}
        0 & a_{22}x\\
        0 & 0
        \end{pmatrix},
        \begin{pmatrix}
        0& b_{22}x\\
        0 &  0
        \end{pmatrix}
        \right]\right|x\in F \right\}
        \;\;\mbox{with}\;\;(a_{22},b_{22})\neq (0,0).
\end{equation}
Conversely, any subset of $T^2$ as in \eqref{eq:X}--\eqref{eq:gamma} is easily
seen to be a submodule of the appropriate type.

\section{Representation of points}\label{se:represent}

The unimodular points of $T^2$ can be represented as subspaces of an $F$-vector
space as follows. Let $U$ be a \emph{faithful} right module over $T$ of
$F$-dimension $r$. (Recall that $F\subset T$ due to our identification.) It
will be convenient to write $x u := u x$ for all $u\in U$ and all $x\in F$. So
elements of $F$ may act on $U$ from either side, whereas proper ternions act
from the right hand side only. Then $U\times U$ is a right $T$-module in the
usual way and at the same time an $F$-vector space of dimension $2r$. The
assignment
\begin{equation}\label{eq:Grass-uni}
    T(A,B) \mapsto \{(uA,uB)\mid u\in U\}
\end{equation}
defines an \emph{injective} map $\Phi_U$, say, of the set of unimodular points
into the Grassmannian $\cG_r(U\times U)$ of $r$-dimensional subspaces of
$U\times U$. This is a direct consequence of more general results from
\cite[Theorem~4.2]{blunck+h-00b} and \cite[pp.~805--806]{herz-95}. Like many
authors we adopt the projective point of view: The elements of the Grassmannian
$\cG_r(U\times U)$ will be identified with the corresponding $(r-1)$-flats
(projective subspaces) of the $(2r-1)$-dimensional projective space on $U\times
U$ (over $F$).

\par
The easiest example is obtained by choosing $U:=F^2$ and by defining the right
action of a ternion on a row $(u_1,u_2)\in F^2$ as the usual matrix
multiplication. It is convenient to identify here the pair
$\big((u_1,u_2),(v_1,v_2)\big)\in U\times U$ with the row vector
$(u_1,u_2,v_1,v_2)\in F^4$. This leads to a well known model for the set of
unimodular points as a set of lines ($1$-flats) in the three-dimensional
projective space over $F$: The images of the unimodular points under
$\Phi_{F^2}$ are precisely those lines of this projective space which meet the
line generated by $(0,1,0,0)$ and $(0,0,0,1)$ at an arbitrary point
($1$-dimensional subspace, $0$-flat). In other words (cf., \emph{e.~g.}, the
table in \cite[p.~30]{hirschfeld-85}), the unimodular points are represented by
the lines of a \emph{special linear complex} without its axis. A proof can be
found in \cite[Example~5.5]{blunck+h-00b} or \cite[p.~239]{blunck+he-05} (up to
a permutation of coordinates due to the usage of lower triangular matrices).
For the reader's convenience let us sketch the easy proof. The point $X$ from
(\ref{eq:X}) is mapped under $\Phi_{F^2}$ to the subspace comprising all
vectors of the form
\begin{equation*}
    \left[
    (u_1,u_2)
        \begin{pmatrix}
        a_{11}x & a_{12}x+a_{22}y\\
        0 & a_{22}z
        \end{pmatrix},
    (u_1,u_2)
        \begin{pmatrix}
         b_{11}x& b_{12}x+b_{22}y\\
         0 & b_{22}z
        \end{pmatrix}
        \right]
\end{equation*}
with variable $u_1,u_2,x,y,z\in F$. This subspace is spanned by the linearly
independent vectors
\begin{equation*}
    (a_{11}, a_{12}, b_{11}, b_{12}) \mbox{\;\;\;\;and\;\;\;\;}
    (0, a_{22}, 0, b_{22}),
\end{equation*}
whence it is a line. The same representation was obtained in several papers on
projective geometry over ternions \cite[p.~157]{beck-36a},
\cite[pp.~128--132]{depu-59}, \cite[Part~C]{depu-60}. The mapping $\Phi_{F^2}$
can be extended to a mapping of non-unimodular points by following
(\ref{eq:Grass-uni}). Indeed, the $Y$-submodule (\ref{eq:Y}) is mapped to the
subspace
\begin{equation*}
    \left[
    (u_1,u_2)
        \begin{pmatrix}
        0 & a_{22}y +c_{22}x\\
        0 & a_{22}z
        \end{pmatrix},
    (u_1,u_2)
        \begin{pmatrix}
         0 & b_{22}y+d_{22}x\\
         0 & b_{22}z
        \end{pmatrix}
        \right]
\end{equation*}
with variable $u_1,u_2,x,y,z\in F$. But this subspace does not depend on the
choice of $Y$, because it is spanned by the vectors $ (0,1,0,0)$ and
$(0,0,1,0)$. Thus, in projective terms, \emph{all\/} non-unimodular points are
mapped to one line, namely the axis of the special linear complex we
encountered before. Hence this extended map is no longer injective and
therefore of little use.

In order to obtain an injective representation of unimodular and non-unimodular
points, we make use of the \emph{right regular representation\/} of $T$, that
is we let $U=T$. As the mapping $\Phi_T$ is the identity, there will be no need
to write it down explicitly below. There is the natural identification $\Phi$
of the free left module $T^{2}$ with the vector space $F^{6}$:
\begin{equation}\label{eq:Phi}
    \left[
    \begin{pmatrix}
    a_{11} & a_{12}\\
    0 & a_{22}
    \end{pmatrix},
    \begin{pmatrix}
    b_{11} & b_{12}\\
    0 & b_{22}
    \end{pmatrix}
    \right ]\stackrel\Phi\longmapsto (a_{11},a_{12},a_{22},b_{11},b_{12},b_{22}).
\end{equation}
We put $x_{1},x_{2},\dots,x_{6}$ for the coordinates in $F^6$. Any
$S\in\GL_2(T)$ defines a linear bijection $g:T^2\to T^2$ and a linear bijection
$f:F^6\to F^6$ which is characterised by $\Phi\circ g=f\circ\Phi$. If $S$ is
given as (\ref{eq:block}) the associated matrix of $f$ has the form
\begin{equation}\label{eq:block6}
    \left(\!\!
    \begin{array}{cc|c|cc|c}
    a_{11} & a_{12} &0 &b_{11} & b_{12}&0\\
    0      & a_{22} &0 &0& b_{22}&0\\
    \hline
    0 &0     & a_{22} &0 &0& b_{22}\\
    \hline
    c_{11} & c_{12} &0 & d_{11} & d_{12}&0\\
    0      & c_{22} &0 & 0      & d_{22}&0\\
    \hline
    0 &0     & c_{22} &0 &0& d_{22}\\
    \end{array}\!\!\right)\in\GL_6(F).
\end{equation}
Denote by ${\cG}_{k}(F^6)=:{\cG}_{k}$ the Grassmannian consisting of
$k$-dimensional subspaces of $F^6$, $k\in \{1,2,3,4,5\}$. Like before the
elements of this Grassmannian will be identified with the corresponding
$(k-1)$-flats of the $5$-dimensional projective space over $F$ (points, lines,
planes, solids, hyperplanes). We write ${\cG}$ for the set formed by the
$\Phi$-images of all non-zero cyclic submodules. For each $i\in
\{X,Y,\alpha,\beta,\gamma\}$ we denote by ${\cG}_{i}$ the set of the
$\Phi$-images of all submodules of type $i$. So,
\begin{equation}\label{eq:cG}
    {\cG}={\cG}_{X}\cup{\cG}_{Y}\cup{\cG}_{\alpha}\cup
    {\cG}_{\beta}\cup{\cG}_{\gamma}.
\end{equation}
Since every element $S\in\GL_{2}(T)$ induces a linear automorphism of $F^6$
according to \eqref{eq:block6}, the $\Phi$-images of submodules of the same
type have the same dimension. From (\ref{eq:X0})--(\ref{eq:gamma0}), we obtain
\begin{displaymath}
    {\cG}_{X}\cup{\cG}_{Y}\subset{\cG}_{3},\;\;\;
    {\cG}_{\alpha}\subset{\cG}_{2},\;\;\;
    {\cG}_{\beta}\cup{\cG}_{\gamma}\subset{\cG}_{1}.
\end{displaymath}
Hence all non-zero cyclic submodules of $T^2$ are represented by non-zero
subspaces of $F^6$. Since $\Phi:T^2\to F^6$ is injective, the following holds
trivially:
\begin{prop}
The $\Phi$-images of distinct cyclic submodules of $T^2$ are distinct subspaces
of $F^6$.
\end{prop}
In particular, the $\Phi$-images of distinct free cyclic submodules are
distinct planes of $F^6$.

\section{Structure of ${\cG}$}

Even though we aim at describing $\cG_X\cup \cG_Y$, \emph{i.~e.}, the
$\Phi$-images of free cyclic submodules of $T^2$, we shall exhibit the entire
set $\cG$ from \eqref{eq:cG}, because the remaining elements of $\cG$ will turn
out useful. First we recall some basic notions for Grassmannians $\cG_k$, $k\in
\{1,2,3,4,5\}$. If $V$ and $W$ are subspaces of $F^6$ with $V\subset W$ then
$[V,W]_k$ denotes the subset of $\cG_k$ formed by all $k$-dimensional subspaces
which contain $V$ and are contained in $W$. If moreover $\dim V=k-1$ and $\dim
W=k+1$ then $[V,W]_k$ is called a \emph{pencil\/} of $\cG_k$.

\par
In our further investigation the solids
\begin{equation}\label{eq:J}
    J \;\;\;\mbox{defined by the conditions}\;\;\; x_3=x_6=0,
\end{equation}
\begin{equation}\label{eq:K}
    K \;\;\;\mbox{defined by the conditions}\;\;\; x_1=x_4=0,
\end{equation}
and their intersection, namely the line
\begin{equation}\label{eq:L}
    L \;\;\;\mbox{defined by the conditions}\;\;\; x_1=x_3=x_4=x_6=0
\end{equation}
will play a crucial role. Furthermore, in the solid $K$ we have the hyperbolic
quadric
\begin{equation}\label{eq:H}
    H\;\;\;\mbox{defined by the conditions}\;\;\; x_2x_6-x_3x_5=x_1=x_4=0.
\end{equation}
$J$ and $K$ (and henceforth $L$) are invariant subspaces of any linear
bijection of $F^6$ which arises from $S\in\GL_2(T) $ according to
\eqref{eq:block6}. This can be read off immediately from the rows of the matrix
in (\ref{eq:block6}). Furthermore, also the hyperbolic quadric $H$ is easily
seen to be invariant under any such linear bijection.

\begin{prop}\label{prop:gamma-beta-alpha}
The following assertions are fulfilled:
\begin{enumerate}
\item ${\cG}_{\gamma}$ coincides with the set of all points of the line
    $L$. The line $L$ is a generator of the hyperbolic quadric $H$.

\item ${\cG}_{\beta}$ coincides with the set of all points of the solid $J$
    which are off the line $L$.

\item ${\cG}_{\alpha}$ is that regulus of the hyperbolic quadric $H$ which
    does not contain $L$.

\end{enumerate}
\end{prop}

\begin{proof}
The first two assertions hold because of \eqref{eq:gamma} and \eqref{eq:beta}.
To show the last assertion we first notice that $\Phi(\gamma_0)=
\Phi(\alpha_0)\cap L$ and that $\Phi(\alpha_0)$ is a line of the hyperbolic
quadric $H$. As $S$ varies in $\GL_2(T)$ the point $\Phi(\gamma_0S)=
\Phi(\alpha_0S) \cap L$ ranges in $\cG_\gamma$, whence the line
$\Phi(\alpha_0S)$ ranges in that regulus on $H$ which does not contain $L$.
\end{proof}

\begin{prop}\label{prop:Y}
${\cG}_{Y}$ is a pencil of planes, namely $[L,K]_3$.
\end{prop}

\begin{proof}
Let $Y$ be a free submodule as in \eqref{eq:Y}. Hence $\Phi(Y)$ is contained in
the solid $K$ defined in \eqref{eq:K}. Letting $z=0$ in \eqref{eq:Y} shows that
the line $L$ is contained in the plane $\Phi(Y)$. So we have ${\cG}_{Y}\subset
[L,K]_{3}$.

Each plane $N\subset K$ satisfies conditions
\begin{displaymath}
    ax_{2}+bx_{3}+cx_{5}+dx_{6}=0,\;\;x_{1}=x_{4}=0
\end{displaymath}
for some $a,b,c,d\in F$ not all $0$. If this plane contains $L$ then
$ax_{2}+cx_{5}=0$ for all $x_{2},x_{5}\in F$. Hence $a=c=0$ and
$(b,d)\neq(0,0)$. We define $a_{22}:=-d$ and $b_{22}:=b$. So there exist
$c_{22},d_{22}\in F$ such that $( a_{22} d_{22}-b_{22} c_{22})\neq 0$. By the
remark after \eqref{eq:X}--\eqref{eq:gamma}, there exists a free submodule $Y$
with $\Phi(Y)=N$. Therefore, ${\cG}_{Y}$ coincides with the pencil $[L,K]_{3}$.
\end{proof}

\begin{prop}\label{prop:X}
${\cG}_{X}$ consist of all planes $M$ of the Grassmannian ${\cG}_{3}$ which
satisfy the following two conditions:
\begin{equation}\label{eq:X-J}
    M \cap J \;\;\;\mbox{is a line other than}\;\;\; L.
\end{equation}
\begin{equation}\label{eq:X-K}
    M \cap K \;\;\;\mbox{is a line belonging to the regulus}\;\;\; \cG_\alpha.
\end{equation}

\end{prop}

\begin{proof}
Let $X$ be a free submodule as in \eqref{eq:X}. Letting $z=0$ in \eqref{eq:X}
shows that $\Phi(X)\cap J$ is a line of $J$ other than $L$. Letting $x=0$ in
\eqref{eq:X} shows that $\Phi(X)\cap K$ is a line belonging to the regulus
$\cG_\alpha$.
\par
Conversely, suppose that $M$ is a plane satisfying \eqref{eq:X-J} and
\eqref{eq:X-K}. There exists a torsion submodule $\beta$ such that
$\Phi(\beta)$ is a point on the line $M\cap L$, and a torsion submodule
$\alpha$ with $\Phi(\alpha)=M\cap K$. These submodules $\beta$ and $\alpha$ can
be written as in \eqref{eq:beta} and \eqref{eq:alpha} which gives coefficients
$a_{11},a_{12},a_{22},b_{11},b_{12},b_{22}\in F$ subject to conditions stated
there. We use these coefficients to define a submodule $X$ according to
\eqref{eq:X}. Then $\Phi(X)=M$.
\end{proof}
By the above proposition, $\cG_{X}$ is formed by all planes spanned by pairs of
lines, where one of them is from $J$ and distinct from $L$, and the other is
from the regulus $\cG_\alpha$.
\par
Recall that flats $P,Q$ are said to be \emph{incident}, in symbols $P\In Q$, if
$P\subset Q$ or $Q\subset P$. The proof of Proposition~\ref{prop:I} below
provides the lengthy description of all incident pairs $(P_0,Q)$ for a fixed
$P_0$ and a variable $Q$.

\begin{prop}\label{prop:I}
The following table displays the number of incident pairs $(P_0,Q)$, where
$P_0\in\cG$ is fixed and $Q\in\cG$ is variable:
\begin{equation*}
    \begin{array}{c|ccccc}
    \In    & Q\in\cG_X        & Q\in\cG_Y& Q\in\cG_\alpha & Q\in\cG_\beta & Q\in\cG_\gamma \\\hline
     P_0\in\cG_X     & 1        & 0    & 1        & |F|  & 1 \\
     P_0\in\cG_Y     & 0        & 1    & 1        & 0    & |F|+1 \\
     P_0\in\cG_\alpha& |F|^2+|F|& 1    & 1        & 0    & 1 \\
     P_0\in\cG_\beta & |F|+1    & 0    & 0        & 1    & 0\\
     P_0\in\cG_\gamma& |F|^2+|F|& |F|+1& 1        & 0    & 1\\
    \end{array}
\end{equation*}
\end{prop}
\begin{proof}
We sketch the proof by completely describing all possibilities.
\par
(a) Let $P_0$ in $\cG_X$. Then $P_0\In Q\in\cG_X\cup\cG_Y$ is equivalent to
$Q=P_0$. $P_0\In Q\in\cG_\alpha$ is equivalent to $Q=P_0\cap K$. $P_0\In
Q\in\cG_\beta$ holds if, and only if, $Q$ is one of the $|F|$ points on
$P_0\cap J$ other than $P_0\cap L$. $P_0\In Q\in\cG_\gamma$ is equivalent to
$Q=P_0\cap L$.
\par
(b) Let $P_0$ in $\cG_Y$. Then $P_0\In Q\in\cG_X\cup\cG_Y$ is equivalent to
$Q=P_0$. $P_0\In Q\in\cG_\alpha$ holds if, and only if, $Q$ is the only
generator other than $L$ of the hyperbolic quadric $H$ which belongs to the
plane $P_0$. $P_0\In Q\in\cG_\beta$ is impossible. $P_0\In Q\in\cG_\gamma$ is
equivalent to $Q$ being one of the $|F|+1$ points on $L$.

\par
(c) Let $P_0$ in $\cG_\alpha$. Then $P_0\In Q\in\cG_X$ is equivalent to $Q$
being one of the $|F|^2+|F|$ planes of $[P_0,P_0+J]_3$ other than $P_0+L$.
$P_0\In Q\in\cG_Y$ is equivalent to $Q=P_0+L$. $P_0\In Q\in\cG_\alpha$ is
equivalent to $Q=P_0$. $P_0\In Q\in\cG_\beta$ is impossible. $P_0\In
Q\in\cG_\gamma$ is equivalent to $Q=P_0\cap L$.

\par
(d) Let $P_0$ in $\cG_\beta$. Then $P_0\In Q\in\cG_X$ is equivalent to $Q$
being one of the $|F|+1$ planes of the form $P_0+R$, where $R$ ranges in the
regulus $\cG_\alpha$. $P_0\In Q\in\cG \setminus \cG_X$ is equivalent to
$Q=P_0$.
\par
\par
(e) Let $P_0$ in $\cG_\gamma$. Then $P_0\In Q\in\cG_\alpha$ is equivalent to
$Q=R_0$, where $R_0$ denotes the only generator other than $L$ of the
hyperbolic quadric $H$ through the point $P_0$. (This line $R_0$ is also used
in the next two subcases.) $P_0\In Q\in\cG_X$ is equivalent to $Q$ being one of
the $|F|^2+|F|$ planes of $[R_0,J+R_0]_3$ other than $R_0+L$. $P_0\In
Q\in\cG_Y$ is equivalent to $Q\in\cG_Y$, which has $|F|+1$ elements. $P_0\In
Q\in\cG_\beta$ is impossible. $P_0\In Q\in\cG_\gamma$ is equivalent to $Q=P_0$.
\end{proof}

\begin{rem}\label{rem:1}
If we associate with each $M\in\cG_X$ the line $M\cap J$ then a line model for
the set of unimodular points is obtained. This is just the special linear
complex mentioned at the beginning of Section~\ref{se:represent} with $L$ being
the axis of the complex.
\end{rem}

\section{Adjacency}

For any $k\in\{1,2,3,4,5 \}$ the flats $Z_1,Z_2\in\cG_k$ are called
\emph{adjacent}, in symbols $Z_1\sim Z_2$, if their intersection is
$(k-1)$-dimensional. It will also be convenient to write $Z_1\cong Z_2$ for
elements which are adjacent or identical. The cases $k=1$ and $k=5$ are
trivial, because any two distinct points and any two distinct hyperplanes are
adjacent, whereas in the remaining cases the entire geometry of the
Grassmannian $\cG_k$ can be based solely on adjacency due to the famous theorem
of Chow. See, \emph{e.~g.}, Chapter~3 in \cite{pankov-10a} for more
information.
\par
Our aim is to exhibit the restriction of the adjacency relation to the sets
$\cG_X$, $\cG_Y$, and $\cG_X\cup\cG_Y$. They represent the sets of unimodular
points, non-unimodular points, and all points of the projective line over $T$.
As before, the flats $J$, $K$, $L$, and the hyperbolic quadric $H$ defined in
\eqref{eq:J}--\eqref{eq:H} will be of great importance. Given $M_1,M_2\in\cG_X$
we have
\begin{equation}\label{eq:adj_X}
    M_1\cong M_2\;\;\; \Leftrightarrow\;\;\; M_1\cap K = M_2\cap K,
\end{equation}
since for $M_1\neq M_2$ either side of \eqref{eq:adj_X} is equivalent to
$M_1+M_2$ being a solid, whereas for $M_1=M_2$ \eqref{eq:adj_X} holds
trivially. We infer from \eqref{eq:adj_X} that $\cong$ is an equivalence
relation on $\cG_X$. The equivalence classes are of the form
\begin{equation}\label{eq:class_X}
    [P,P+J]_3\setminus\{P+L\}\mbox{~~with~~} P\in\cG_\alpha.
\end{equation}
So there is a one-one correspondence between these equivalence classes and the
lines of the regulus $\cG_\alpha$.
\par
We now consider the adjacency relation on $\cG_X\cup\cG_Y$. Each plane
$M\in\cG_X$ is adjacent with precisely one plane of $\cG_Y$, namely
\begin{equation*}
    N:=(M\cap K)+L.
\end{equation*}
As $\cG_Y$ is a pencil of planes, we have $N_1\cong N_2$ for all
$N_1,N_2\in\cG_Y$. Consequently, the pencil $[L,K]_3$ and the subsets
\begin{equation}\label{eq:cliques_XY}
    [P,P+J]_3\;\;\mbox{with}\;\; P\in\cG_\alpha
\end{equation}
are the \emph{cliques\/} of $\cG_X\cup\cG_Y$ with respect to adjacency. That
means, these are the maximal subsets of $\cG_X\cup\cG_Y$ for which any two
distinct elements are adjacent. It is well known that the elements of any
clique \eqref{eq:cliques_XY} and the pencils contained in it can be considered
as the ``points'' and ``lines'' of a projective plane. Thus any equivalence
class from \eqref{eq:class_X} can be regarded as a \emph{punctured projective
plane}, \emph{i.~e.}, a projective plane with one point removed. Going over
from $\cG_X$ to $\cG_X\cup\cG_Y$ provides the ``closure'' of all these
punctured projective planes.
\par

Another advantage of $\cG_X\cup\cG_Y$ over $\cG_X$ is \emph{connectedness\/}
with respect to $\sim$, \emph{i.~e.}, given any $Z,Z'\in \cG_X\cup \cG_Y$ there
exists a finite sequence $Z_0, Z_1,\ldots, Z_r$ of planes from $\cG_X\cup\cG_Y$
such that
\begin{equation*}
    Z=Z_0\sim Z_1\sim\cdots\sim Z_r = Z'.
\end{equation*}
The minimal $r\geq 0$ for which such a sequence exists is said to be the
\emph{distance} between $Z$ and $Z'$. Any two \emph{distinct\/} planes
$M_1,M_2\in \cG_X$ are either at distance $1$ (adjacent) or at distance $3$,
because for $M_1\not\sim M_2$ we have
\begin{equation*}
    M_1\sim (M_1\cap K)+L \sim (M_2\cap K)+L \sim M_2,
\end{equation*}
and this is the only shortest sequence from $M_1$ to $M_2$. Yet, this distance
function is of restricted use, since in the latter case $M_1\cap M_2$ may be
$0$ (the ``empty flat'') or a single point.
\par
Any bijection of $\cG_X\cup\cG_Y$ onto itself which preserves adjacency in both
directions will be called an \emph{adjacency preserver} of $\cG_X\cup\cG_Y$.
The straightforward proof of the following result is left to the reader:

\begin{prop}\label{prop:adj}
Let $\mu:\cG_\alpha\to\cG_\alpha$ be a bijection. For each $P\in\cG_\alpha$ we
choose a bijection
\begin{equation}\label{eq:psi_P}
    \psi_P: [P,P+J]_3 \to [\mu(P),\mu(P)+J]_3 \;\;\;\;\mbox{such that}\;\;\;\;
    P+L  \mapsto \mu(P)+L.
\end{equation}
Then the mapping
\begin{equation}\label{eq:lambda}
    \lambda:\cG_X\cup\cG_Y\to\cG_X\cup\cG_Y : Z \mapsto \psi_P(Z)
    \;\;\;\;\mbox{if}\;\;\;\;Z\in [P,P+J]_3
\end{equation}
is a well-defined adjacency preserver of $\cG_X\cup\cG_Y$. Conversely, every
adjacency preserver of $\cG_X\cup\cG_Y$ arises in this way.
\end{prop}
Due to the condition on $\psi_P$ stated at the end of formula \eqref{eq:psi_P}
we obtain the following:
\begin{cor}\label{cor:1}
Every adjacency preserver $\lambda$ of $\cG_X\cup\cG_Y$ satisfies
$\lambda(\cG_X)=\cG_X$ and $\lambda(\cG_Y)=\cG_Y$.
\end{cor}

If we impose the much stronger requirement that a bijection of $\cG_X\cup\cG_Y$
onto itself should map pencils of planes onto pencils of planes then all
bijections from \eqref{eq:psi_P} have to be collineations between the
underlying projective planes, but there need not be any relationship between
these collineations. So even here we obtain transformations as in
\eqref{eq:lambda} that do not really deserve our interest.
\par
Still, we have the following description of adjacency preservers which arise
from semilinear bijections. Recall that via the bijection \eqref{eq:Phi} the
points of the projective line over $T$ correspond to the planes belonging to
$\cG_X\cup\cG_Y$.

\begin{theorem}\label{thm:1}
For all mappings $g:T^2\to T^2$ and all mappings $f:F^6\to F^6$ such that
$\Phi\circ g = f\circ \Phi$ the following statements are equivalent.
\begin{enumerate}
  \item[(i)] $g$ is a $T$-semilinear bijection of the module $T^2$.
  \item[(ii)] $f$ is an $F$-semilinear bijection of the vector space $F^6$
      satisfying $f(\cG_X\cup\cG_Y)=\cG_X\cup\cG_Y$.
    \item[(iii)] $f$ is an $F$-semilinear bijection of the vector space
        $F^6$ satisfying $f(\cG_X)=\cG_X$.
  \item[(iv)] $f$ is an $F$-semilinear bijection of the vector space $F^6$
      satisfying $f(J)=J$ and $f(H)=H$.
\end{enumerate}
\end{theorem}

We postpone the proof until we have established two lemmas.

\begin{lemma}\label{lem:line}
A line $Q$ has the property that for all $M\in\cG_X$ the intersection $Q\cap M$
is a point if, and only if, $Q$ belongs to the opposite regulus of
$\cG_\alpha$.
\end{lemma}
\begin{proof}

(a) Let $Q$ be a line such that $Q\cap M$ is a point for all $M\in\cG_X$. We
first choose a fixed $P\in\cG_\alpha$ and define the hyperplane $W:=P+J$. Then
all but one planes of $[P,W]_3$ belong to $\cG_X$; the only exception is
$P+L\in\cG_Y$. We claim that $Q\cap P$ contains a point. Assume to the contrary
that this were not the case. So, for all $M\in[P,W]_3\cap\cG_X$, the point
$Q\cap M$ would be off the line $P$. This would imply that all such $M$ could
be written as $M=P+(M\cap Q)$, whence all of them would be contained in the
solid $P+Q$, a contradiction.
\par
Now, as $P$ ranges in the regulus $\cG_\alpha$, we see that $Q$ is a line of
the regulus opposite to $\cG_\alpha$.
\par
(b) The converse is obviously true.
\end{proof}

\begin{lemma}\label{lem:solid}
A solid $V$ has the property that for all $M\in\cG_X$ the intersection $V\cap
M$ is a line if, and only if, either $V=J$ or $V=K$.
\end{lemma}
\begin{proof}
(a) Let $V$ be a solid such that $V\cap M$ is a line for all $M\in\cG_X$. We
first choose a fixed $P\in\cG_\alpha$ and define the hyperplane $W:=P+J$. Then
all but one planes of $[P,W]_3$ belong to $\cG_X$; the only exception is
$P+L\in\cG_Y$. We claim that
\begin{equation}\label{eq:iff}
    V\subset W \;\;\Leftrightarrow\;\; P\not\subset V.
\end{equation}
On the one hand, for $V\subset W$ we obtain $P\not\subset V$, since otherwise
there would exist an $M\in [P,V]_3\cap\cG_X$, and $M\cap V=M$ would not be a
line. On the other hand, for $V\not\subset W$ the intersection $V\cap W$ is a
plane. We have $P\subset V$, since otherwise there would exist an $M\in
[P,W]_3\cap\cG_X$ with $M\cap(V\cap W)=M\cap V$ being a point. For the rest of
the proof we distinguish two cases:
\par
Case 1: There exists a line of $\cG_\alpha$, say $P_1$, with $P_1\subset V$.
There exist two further lines $P_2,P_3\in \cG_\alpha$ and three mutually skew
planes $M_1,M_2,M_3\in\cG_X$ with $M_i\cap K=P_i$ for $i\in\{1,2,3\}$. Through
each point of the plane $M_1$ there is a unique line which meets each of the
planes $M_2$ and $M_3$ at a point. Within the solid $V$ we have a similar
result about the mutually skew lines $M_1\cap V, M_2\cap V, M_3\cap V$: Through
each point of the line $M_1\cap V=P_1$ there is a unique line which meets each
of the lines $M_2\cap V$ and $M_3\cap V$ at a point. Furthermore, through each
point of $P_1$ there is a unique line of the regulus which is opposite to
$\cG_\alpha$, and this line meets each of the lines $P_2\subset M_2$ and
$P_3\subset M_3$ at a point. We combine these three observations and infer that
all lines of the opposite regulus of $\cG_\alpha$ are contained in $V$, that
$P\subset V$ for all $P\in\cG_\alpha$, and finally that $V=K$.
\par
Case 2: For all $P\in\cG_\alpha$ holds $P\not\subset V$. From \eqref{eq:iff}
(with $W$ to be replaced by $P+J$) we infer $V\subset P+J$ for all $P\in\cG_X$.
Due to $\bigcap_{P\in\cG_X} (P+J) = J$, now there holds $V=J$.
\end{proof}

\begin{proof}[Proof of Theorem~\emph{\ref{thm:1}}]

(i) $\Rightarrow$ (ii): The accompanying automorphism of $g$ is, like any
automorphism of $T$, the product of an automorphism $\sigma$ of $F$ (acting
entrywise on $T$) followed by an inner automorphism of $T$
\cite[Theorem~6.6]{chooi+lim-02a}. By virtue of this result, a straightforward
calculation shows that $f=\Phi\circ g\circ \Phi^{-1}$ is an $F$-semilinear
bijection with respect to $\sigma$. The assertion $f(\cG_X)=\cG_X$ is obviously
true.
\par
(ii) $\Rightarrow$ (iii): The $F$-semilinear bijection $f$ is an adjacency
preserver on $\cG_X\cup \cG_Y$. By Corollary~\ref{cor:1} we have
$f(\cG_X)=\cG_X$.
\par
(iii) $\Rightarrow$ (iv): We infer from Lemma~\ref{lem:line} that the regulus
opposite to $\cG_\alpha$ is fixed, as a set of lines, under $f$. Therefore also
the hyperbolic quadric $H$ and the solid $K$ are invariant under $f$. From
$f(K)=K$ and Lemma~\ref{lem:solid} follows $f(J)=J$.
\par
(iv) $\Rightarrow$ (i): The $F$-semilinear bijection $f$ can be written as a
product of three $F$-semilinear bijections $f_1, f_2, f_3$ as follows:
\par
The first bijection is $f_1:(x_1,x_2,\ldots x_6)\mapsto
(\sigma(x_1),\sigma(x_2),\ldots \sigma(x_6))$, where $\sigma$ is the
automorphism of $F$ accompanying $f$. Then $g_1:=\Phi^{-1}\circ f_1\circ\Phi$
is clearly a $T$-semilinear bijection of $T^2$.
\par
The second mapping is the unique linear bijection $f_2:F^6\to F^6$ which fixes
all vectors of $J$ and permutes the lines of the regulus opposite to
$\cG_\alpha$ in the same way as $f\circ f_1^{-1}$. All lines of $\cG_\alpha$
are two-dimensional $f_2$-invariant subspaces. The matrix of $f_2$ can
therefore be written in block diagonal form as $\diag(1,G,1,G)$ with
\begin{equation*}
     G:=\begin{pmatrix}
        1&0\\
        a&b\\
     \end{pmatrix}\in \GL_2(F).
\end{equation*}
The transpose $G^t$ is a ternion. The homothety $g_2:T^2\to T^2: (A,B)\mapsto
G^t (A,B)$ is $T$-semilinear (with respect to an inner automorphism) and an
easy calculation shows $\Phi\circ g_2=f_2\circ\Phi$.
\par
The third mapping is the linear bijection $f_3:=f\circ f_1^{-1}\circ f_2^{-1}$.
It has $J$, $K$ and all lines of the regulus opposite to $\cG_\alpha$ as
invariant subspaces. This implies that the matrix of $f_3$ can be written as in
\eqref{eq:block6}, whence the invertible matrix \eqref{eq:block} defines
$g_3:T^2\to T^2$ with the property $\Phi\circ g_3=f_3\circ\Phi$.
\par
Finally, from $f=f_3\circ f_2\circ f_1$ follows that $g=g_3\circ g_2\circ g_1$
is a $T$-semilinear bijection.
\end{proof}

The previous theorem describes all automorphic collineations of $\cG_X\cup
\cG_Y$ and also of $\cG_X$. It is in the spirit of results from
\cite{westwick-64a}, \cite{westwick-69a}, and \cite{westwick-74a}.
\par
A duality of the projective space on $F^6$ maps any Grassmannian $\cG_k$,
$k\in\{1,2,3,4,5\}$, bijectively onto the Grassmannian $\cG_{6-k}$, and it
preserves adjacency in both directions. In particular, $\cG_3$ is mapped onto
itself. At the first sight somewhat surprisingly, the following holds:
\begin{theorem}
Neither $\cG_X\cup \cG_Y$ nor $\cG_X$ is fixed, as a set of planes, under a
duality of the projective space on $F^6$.
\end{theorem}

\begin{proof}
Let $\delta$ be a duality fixing $\cG_X\cup \cG_Y$. Hence the restriction of
$\delta$ to $\cG_X\cup \cG_Y$ is an adjacency preserver. By
Corollary~\ref{cor:1}, we obtain $\delta(\cG_X)=\cG_X$.
\par
Assume to the contrary that there exists a duality fixing $\cG_X$. This duality
maps the lines $Q$, described in Lemma~\ref{lem:line}, to the solids $V$,
described in Lemma~\ref{lem:solid}, in a bijective way. We obtain a
contradiction, since there are $|F|+1>2$ such lines, but only two such solids.
\end{proof}
\begin{rem}
Every \emph{Jordan automorphism\/} of an arbitrary ring $R$ of stable rank $2$
defines a bijective mapping of the set of unimodular points of $R^2$ onto
itself. See \cite[4.2]{blunck+he-05} or \cite[p.~832]{herz-95} for further
details and, in particular, the rather complicated definition of such a
mapping. For our $F$-algebra of ternions the situation is less intricate, since
every Jordan automorphism of $T$ is either an automorphism or an
antiautomorphism; moreover, it is $F$-semilinear
\cite[Theorem~6.6]{chooi+lim-02a}. We already referred to this result in the
proof of Theorem~\ref{thm:1}, where we described all automorphisms of $T$. We
add, for the sake of completeness, that an antiautomorphism of $T$ is given by
\begin{equation*}
    \begin{pmatrix} x & y\\0 & z\end{pmatrix}
    \mapsto
    \begin{pmatrix} z & y\\0 & x\end{pmatrix}.
\end{equation*}
The mappings of unimodular points arising from antiautomorphisms appear also in
\cite{benz-77} for the ternions over the real numbers and in
\cite{blunck+h-01b} in a more general setting.
\par
In terms of the line model from the beginning of Section~\ref{se:represent} any
mapping $\omega$ on unimodular points arising from an antiautomorphism of $T$
is induced by a duality which preserves the axis of the special linear complex.
All lines through one of the points of this axis go over to all lines in one of
the planes through the axis. From Remark~\ref{rem:1} and
Proposition~\ref{prop:X}, in our model we obtain from $\omega$ a bijection
$\xi$ of $\cG_X$ onto itself which \emph{does not preserve adjacency}. Indeed,
there are $M_1,M_2\in\cG_X$ with $M_1\cap M_2\in\cG_\alpha$ for which the lines
$M_1\cap K$, $M_2\cap K$, and $L$ are not coplanar. Therefore $\xi(M_1)\cap
\xi(M_2)\in\cG_\beta$ is only a single point. However, $\xi$ preserves
(unordered) pairs of skew planes, since these correspond via $\Phi$ to
\emph{distant\/} unimodular points (in the terminology of \cite{blunck+he-05}
and \cite{herz-95}) and $\omega$ preserves (unordered) pairs of distant points.
\par
Also, there does not seem to be a natural extension of $\omega$ to
non-unimodular points. This is in sharp contrast to the observation in the
introduction of \cite{faure-04a} that non-unimodular points are
``indispensable'' for an arbitrary semilinear mapping to define a morphism of
projective spaces over rings.
\end{rem}

\paragraph{\textbf{Acknowledgements.}} This work was carried
out within the framework of the Scientific and Technological Cooperation
Poland-Austria 2010--2011.




\par$\;$\par 

\noindent
Hans Havlicek\\
Institut f\"{u}r Diskrete Mathematik und Geometrie\\
Technische Universit\"{a}t\\
Wiedner Hauptstra{\ss}e 8--10/104\\
A-1040 Wien\\
Austria\\
\texttt{havlicek@geometrie.tuwien.ac.at}
\par~\par
\noindent%
Andrzej Matra\'{s} and Mark Pankov\\
Faculty of Mathematics and Computer Science\\
University of Warmia and Mazury\\
\.{Z}o{\l}nierska 14A\\
PL-10-561 Olsztyn\\
Poland\\
\texttt{matras@uwm.edu.pl}\\
\texttt{markpankov@gmail.com}

\end{document}